\documentclass[12pt]{amsart}
\usepackage{bbding}
\usepackage{dsfont}
\usepackage{graphicx}
\usepackage{amsthm,amscd}
\usepackage{calc}
\usepackage{mathrsfs}
\usepackage{CJK,fancyhdr,amscd}
\usepackage{anysize}
\usepackage{amsmath,amssymb,amsfonts}
\usepackage{epsfig}
\usepackage{latexsym}
\usepackage{indentfirst, latexsym}
\usepackage{graphics}
\usepackage{hyperref}

\providecommand{\U}[1]{\protect\rule{.1in}{.1in}}

\numberwithin{equation}{section}
\providecommand{\U}[1]{\protect\rule{.1in}{.1in}}
\newtheorem{theorem} {Theorem} [section]
\newtheorem{proposition}[theorem]{Proposition}
\newtheorem{corollary}  [theorem]     {Corollary}
\newtheorem{lemma}  [theorem]     {Lemma}

\newtheorem{definition}  [theorem]     {Definition}

\newcommand{\dz}{d \bar{z}}

\newcommand{\pz}{\partial\bar{z}}

\newcommand{\w}{\wedge}

\newcommand{\1}{\mathds{1}}

\newcommand{\db}{\overline{\partial}}

\newcommand{\lc}{\lrcorner}

\newcommand{\lk}{\left(}
\newcommand{\rk}{\right)}

\newcommand{\btheorem}{\begin{theorem}}
\newcommand{\etheorem}{\end{theorem}}
\newcommand{\bproposition}{\begin{proposition}}
\newcommand{\eproposition}{\end{proposition}}
\newcommand{\bdefinition}{\begin{definition}}
\newcommand{\edefinition}{\end{definition}}
\newcommand{\bcorollary}{\begin{corollary}}
\newcommand{\ecorollary}{\end{corollary}}
\newcommand{\bproof}{\begin{proof}}
\newcommand{\eproof}{\end{proof}}
\newcommand{\beq}{\begin{equation}}
\newcommand{\eeq}{\end{equation}}
\newcommand{\ee}{\end{eqnarray*}}
\newcommand{\be}{\begin{eqnarray*}}

\newcommand{\elemma}{\end{lemma}}
\newcommand{\blemma}{\begin{lemma}}

\newcommand{\om}{\omega}
\renewcommand{\>}{\rightarrow}

\newcommand{\p}{\partial}

\newcommand{\bd}{\begin{enumerate} }
\newcommand{\ed}{\end{enumerate}}

\begin{document}
\title{Extension formulas and deformation invariance of Hodge numbers}
\author{Quanting Zhao}
\address{School of Mathematics and statistics,
Central China Normal University, Wuhan 430079, China; Center of
Mathematical Sciences, Zhejiang University, Hangzhou 310027, China}
\email{zhaoquanting@126.com}

\author{Sheng Rao}
\address{School of Mathematics and statistics, Wuhan  University,
Wuhan 430072, China; Department of Mathematics, University of
California at Los Angeles, CA 90095-1555, USA}
\email{likeanyone@whu.edu.cn; likeanyone@math.ucla.edu}
\thanks{Rao is the corresponding author.}

\date{\today}


\begin{abstract}
We introduce a canonical isomorphism from the space of pure-type
complex differential forms on a compact complex manifold to the one
on its infinitesimal deformations. By use of this map, we generalize
an extension formula in a recent work of K. Liu, X. Yang and the
second author. As a direct corollary of the extension formulas, we
prove several deformation invariance theorems for Hodge numbers on
some certain classes of complex manifolds, without use of
Fr\"{o}licher inequality or the topological invariance of Betti
numbers.
\end{abstract}
\maketitle

\section{Introduction and main results}
This paper is to study the deformation invariance of Hodge numbers
and we use an iteration method to construct explicit extension of
Dolbeault cohomology classes.


Let $\pi: \mathfrak{X} \> \bigtriangleup$ be a holomorphic family of
$n$-dimensional compact complex manifolds with the central fiber
$\pi^{-1}(0) = X_0$ and its infinitesimal deformations
$\pi^{-1}(t)=X_t$, where $\bigtriangleup$ is a small disk in
$\mathbb{C}$ for simplicity. Then there exists a transversely
holomorphic trivialization $F_{\sigma}: \mathfrak{X}
\stackrel{(\sigma,\pi)}{\longrightarrow} X_0 \times \bigtriangleup$
(cf. \cite[Proposition 9.5]{V} and \cite[Appendix A]{C}), which
gives us the Kuranishi data $\varphi(t)$ (or $\varphi$), depending
holomorphically on $t$, with the integrability
\begin{equation}\label{int}
\db \varphi(t)= \frac{1}{2}[\varphi(t), \varphi(t)].
\end{equation}
Fix an open coordinate covering $\big\{ \mathfrak{U}:
(w^{\alpha}_j,t) \in U^{\alpha}, U^{\alpha} \in \mathfrak{U} \big\}$
of $\mathfrak{X}$, with a restricted covering $\big\{
\mathfrak{U}_0: z_{j}^{\alpha} \in U_0^{\alpha}:= U^{\alpha} \bigcap
X_0, U^{\alpha} \bigcap X_0 \in \mathfrak{U}_0 \big\}$ of $X_0$. As
we focus on one coordinate chart, the superscript $\alpha$ is
suppressed. As in \cite{C,LSY,lry}, the operator $e^{i_{\varphi}}$
is defined by
$$e^{i_{\varphi}}= \sum_{k=0}^{\infty}\frac{1}{k!}i^k_{\varphi},$$
where $i^k_{\varphi}$ denotes $k$ times of the contraction operator
$i_{\varphi}=\varphi \lc$ and $e^{i_{\overline{\varphi}}}$ is
similarly defined. It is known that $\{e^{i_{\varphi}}\lk dz^{i} \rk
\}_{i=1}^n$ and $ \{ \overline{ e^{i_{\varphi}}\lk dz^{i} \rk }
\}_{i=1}^n$ are the local bases of $T^{*(1,0)}_{X_t}$ and
$T^{*(0,1)}_{X_t}$, respectively. Inspired by these, we introduce:
\begin{definition}
A canonical map between $A^{p,q}(X_0)$ and $A^{p,q}(X_t)$ is defined
as:
\[ \begin{array}{cccc}
e^{i_{\varphi}|i_{\overline{\varphi}}}: & A^{p,q}(X_0) & \rightarrow & A^{p,q}(X_t) \\
& \om & \mapsto & e^{i_{\varphi}|i_{\overline{\varphi}}}
\lk \om \rk, \\
\end{array}\]
where \[ e^{i_{\varphi}|i_{\overline{\varphi}}} \lk \om \rk =
\sum_{\begin{subarray}{c} i_1,\cdots,i_p\\
j_1,\cdots,j_q \end{subarray}}
\frac{1}{p!q!}\om_{i_1,\cdots,i_p;j_1,\cdots,j_q}(z)
\left(e^{i_{\varphi}}\left(dz^{i_1}\wedge\cdots\wedge
dz^{i_p}\right)\right) \wedge
\left(e^{i_{\overline{\varphi}}}\left(d\overline{z}^{j_1}\cdots\wedge
d\overline{z}^{j_q}\right)\right)\] and $\om$ is locally written as
$\sum_{\begin{subarray}{c} i_1,\cdots,i_p\\
j_1,\cdots,j_q\\
\end{subarray}} \frac{1}{p!q!}\om_{i_1,\cdots,i_p;j_1,\cdots,j_q}(z) dz^{i_1}\w \cdots \w dz^{i_p}
\w \dz^{j_1} \w \cdots \dz^{j_q}$. \end{definition} It is easy to
check that $e^{i_{\varphi}|i_{\overline{\varphi}}}$  is independent
of the choice of local coordinates and is actually a real
isomorphism. From the explicit formula of $\varphi$ (cf. \cite[pp.
150]{MK}), a careful calculation yields: \blemma\label{dw-dz}
$$\begin{cases} dw^{\alpha} &= \frac{\p
w^{\alpha}}{\p z^i} \lk e^{i_{\varphi}} \lk dz^i \rk \rk \\
\frac{\p\ }{\p w^{\alpha}} &= \lk \lk \1 - \varphi
\overline{\varphi} \rk^{-1} \lk \frac{\p w}{\p z} \rk^{-1}
\rk^j_{\alpha} \frac{\p\ }{\p z^j} - \lk \lk \1- \overline{\varphi}
\varphi \rk^{-1} \overline{ \varphi } \lk \frac{\p w}{\p z} \rk^{-1}
\rk^{j}_{\alpha} \frac{\p\ }{\pz^j},
\end{cases}$$
where $\overline{\varphi}\varphi:=\varphi \lc \overline{\varphi}$
and $\varphi \overline{\varphi}$ is similarly defined. \elemma
\bcorollary\label{pw-pz} $\frac{\p w^{\alpha}}{\p z^i} \frac{\p\
}{\p w^{\alpha}} = \lk \lk \1 - \varphi \overline{\varphi}
\rk^{\!\!-1} \rk^j_i \frac{\p\ }{\p z^j} - \lk \lk \1-
\overline{\varphi} \varphi \rk^{\!\!-1} \overline{ \varphi }
\rk^j_{i} \frac{\p\ }{\pz^j}$. \ecorollary Then we get the following
useful local formula:
\blemma\label{dvz}
\begin{align*} d \lk e^{i_{\varphi}} \lk dz^i \rk \rk &= \lk \lk \1-
\overline{\varphi} \varphi \rk^{\!\!-1} \overline{ \varphi }
\rk^{\bar l}_{\bar k} \frac{\p \varphi^i_{\bar l}}{\p z^j} \lk
e^{i_{\varphi}} \lc dz^k \rk \wedge\lk
e^{i_{\varphi}} \lc dz^j \rk  \\
& \quad - \lk \lk \1- \overline{\varphi} \varphi \rk^{\!\!-1}
\rk^{\bar l}_{\bar k} \frac{\p \varphi^i_{\bar l}}{\p z^j} \lk
\overline{e^{i_{\varphi}} \lc dz^k} \rk \wedge\lk e^{i_{\varphi}}
\lc dz^j \rk, \end{align*} which describes the $d$-operator under
the local frames $\{ e^{i_{\varphi}} \lk dz^i \rk,
\overline{e^{i_{\varphi}} \lk dz^i \rk} \}_{i=1}^n$. \elemma
 Using these,
one has: \bproposition\label{f-extension} Let $f$ be a smooth
function on $X_0$. Then \[ df=
e^{i_{\varphi}|i_{\overline{\varphi}}} \Big( \lk \1-\varphi
\overline{\varphi}\rk^{\!\!-1}\lc(\p- \overline{\varphi}\lc \db)f +
\lk \1-\overline{\varphi}\varphi\rk^{\!\!-1}\lc(\db-\varphi\lc \p)f
\Big). \] \eproposition Since $df$ can be decomposed into $\p_t f +
\db_t f$ on $X_t$, $\db_t f =
e^{i_{\varphi}|i_{\overline{\varphi}}}\Big( \lk
\1-\overline{\varphi}\varphi\rk^{\!\!-1}\lc(\db-\varphi\lc \p)f
\Big)$. Thus $f$ is holomorphic with respect to the complex
structure of $X_t$ if and only if $$(\db-\varphi\lc \p)f=0,$$ by the
invertibility of $(\1-\overline{\varphi}\varphi)^{\!\!-1}\lc$.
Hence, we reprove this  important criterion (cf. \cite{nn} and also
\cite[pp. 151-152]{MK}) in the deformation theory.

Then we get two extension formulas on $(p,0)$ and $(0,q)$-forms.%
\begin{proposition}\label{p0-extension}
For $\omega\in A^{p,0}(X_0)$,
\begin{align*}
   d(e^{i_{\varphi}|i_{\overline{\varphi}}} \lk \omega \rk )
 &=e^{i_{\varphi}|i_{\overline{\varphi}}}\Big(\lk \1 - \overline{\varphi} \varphi \rk^{\!\!-1}\lc(\db\om+\p(\varphi\lc\omega)-\varphi\lc \p\omega)\\
   & \quad \quad \quad \quad +\lk \1 - \varphi \overline{\varphi} \rk^{\!\!-1}\lc\p \omega-p\p
 \omega-(\lk \1 - \varphi \overline{\varphi}
 \rk^{\!\!-1}\lc\overline{\varphi})\lc(\db  \omega+\p
 (\varphi\lc\omega))\Big).\\
\end{align*}
\end{proposition}

\begin{corollary}\label{0q-extension}
For $\omega\in A^{0,q}(X_0)$,
\begin{align*}
   d(e^{i_{\varphi}|i_{\overline{\varphi}}}\lk \omega \rk )=&e^{i_{\varphi}|i_{\overline{\varphi}}}\Big(\lk \1 - \varphi\overline{\varphi} \rk^{\!\!-1}\lc(\p\om+\db(\overline{\varphi}\lc\omega)-\overline{\varphi}\lc \db\omega)\\
   & \quad  \quad \quad+\lk \1 - \overline{\varphi}\varphi
   \rk^{\!\!-1}\lc\db\omega-q\db
 \omega-(\lk \1 -\overline{\varphi} \varphi
 \rk^{\!\!-1}\lc{\varphi})\lc(\p  \omega+\db
 (\overline{\varphi}\lc\omega))\Big).
\end{align*}
\end{corollary}

Based on these two, we use the iteration method, initiated by
\cite{LSY} and developed in \cite{s,sy,lry,lzr,Rz}, to achieve two
theorems on deformation invariance of Hodge numbers, by constructing
explicit extension, without use of Fr\"{o}licher inequality or the
topological invariance of Betti numbers (cf. \cite[Section 5.1]{gr}
and \cite[Section 9.3.2]{V}). We need:

\begin{definition} Define a complex manifold
$X \in \mathcal{E}^{p,q}$, $\mathfrak{D}^{p,q}$ and
$\mathfrak{B}^{p,q}$, if for any $\db$-closed $\p g\in A^{p,q}(X)$,
the equation \[ \db x=\p g \] has a solution, a $\p$-closed solution
and a $\p$-exact solution, respectively. It is obvious that
$\mathfrak{B}^{p,q}\subset\mathfrak{D}^{p,q}\subset\mathcal{E}^{p,q}$
and that $X$, satisfying the $\p\db$-lemma, lies in
$\mathfrak{B}^{p,q}$.
\end{definition}
Set $h^{p,q}_t=dim_{\mathbb{C}}H^{p,q}(X_t, \mathbb{C})$. Then:
\begin{theorem}\label{p0} For $1 \leq p \leq n$ and $X_0 \in \mathfrak{D}^{p,1}\cap
\mathcal{E}^{p+1,0}$, $h^{p,0}_t$  are independent of $t$.
\end{theorem}

\begin{theorem}\label{0q}
For $1\leq q\leq n$ and $X_0 \in
\mathfrak{B}^{1,q'}\cap\mathcal{E}^{q',0}\cap\mathfrak{D}^{q',1}$
with all $1\leq q'\leq q$, $h^{0,q}_t$  are independent of $t$.
\end{theorem}

By Theorem \ref{p0} and the standard Hodge theory on compact complex
surfaces (such as Section IV.2 of \cite{ccs}), we obtain:
\begin{corollary}
All the Hodge numbers of a compact complex surface are infinitesimal
deformation invariant.
\end{corollary}
For the jumping phenomenon of Hodge numbers we refer to \cite{n,y}.
More generally than Proposition \ref{p0-extension} and Corollary
\ref{0q-extension}, we achieve: \bproposition\label{extension-p,q}
For $\omega \in A^{\star,\star}(X_0)$,
\begin{align*}  &d\lk e^{i_{\varphi}|i_{\overline{\varphi}}}
(\om) \rk\\
=&e^{i_{\varphi}|i_{\overline{\varphi}}} \left( \p \om + \lk \varphi
\overline{\varphi}(\1 - \varphi \overline{\varphi})^{\!\!-1} \rk \lc
\p \om - \p \lk \lk \varphi \overline{\varphi} (\1 - \varphi
\overline{\varphi})^{\!\!-1} \rk \lc \om \rk + \lk \p \lk
\overline{\varphi}( \1 - \varphi \overline{\varphi})^{\!\!-1} \rk
\lc \varphi \rk \lc \om \right. \\
& \quad \quad \quad \quad - \lk \overline{\varphi}(\1 - \varphi
\overline{\varphi})^{\!\!-1} \rk \lc \db \om + \db \lk \lk
\overline{\varphi} ( \1 - \varphi \overline{\varphi})^{\!\!-1} \rk
\lc \om \rk - \lk \db (\1-\varphi \overline{\varphi})^{\!\!-1} \lc \overline{\varphi} \rk \lc \om \\
& \quad \quad \quad +\db \om + \lk \overline{\varphi} \varphi(\1 -
\overline{\varphi} \varphi)^{\!\!-1} \rk \lc \db \om - \db \lk
\lk\overline{\varphi} \varphi (\1 -\overline{\varphi}
\varphi)^{\!\!-1} \rk \lc \om \rk + \lk \db \lk \varphi( \1 -
\overline{\varphi} \varphi)^{\!\!-1} \rk
\lc \overline{\varphi} \rk \lc \om  \\
& \left. \quad \quad \quad \quad - \lk \varphi(\1 -
\overline{\varphi} \varphi)^{\!\!-1} \rk \lc \p \om + \p \lk \lk
\varphi ( \1 - \overline{\varphi} \varphi)^{\!\!-1} \rk \lc \om \rk
- \lk \p (\1- \overline{\varphi} \varphi)^{\!\!-1} \lc \varphi \rk
\lc \om \right).
\end{align*} \eproposition
More details and applications (especially for Proposition
\ref{extension-p,q}) will appear in \cite{RZ14}.

\section{The ideas of proofs}
We shall describe the main ideas in the proofs of Theorems \ref{p0}
and \ref{0q} in this section. Throughout this section, $X_t$ is
assumed to be determined by the integrable Kuranishi data
$\varphi(t)=\sum_{k=1}^\infty t^k\varphi_k$ with \eqref{int}.
Theorem \ref{p0} is obtained by Kodaira-Spencer's upper
semi-continuity theorem and the following iteration procedure.
\begin{proposition}\label{thmp0} Let
$X_0\in \mathfrak{D}^{p,1}\cap \mathcal{E}^{p+1,0}$. Then for any
holomorphic $(p,0)$-form $\sigma_0$ on ${X_0}$, there exits a power
series
$$\sigma_t=\sigma_0+\sum_{k=1}^\infty t^k\sigma_k \in
A^{p,0}(X_0),$$ such that $e^{i_{\varphi(t)}}(\sigma_t)\in
A^{p,0}(X_t)$ is holomorphic with respect to the complex structure
on $X_t$.
\end{proposition}
\begin{proof}[Sketch of Proof] By Grauert's formal function theorem \cite{g},
we only need to construct $\sigma_t$ order by order. Proposition
\ref{p0-extension} yields that the holomorphicity of
$e^{i_{\varphi(t)}} \lk \sigma_t \rk$ is equivalent to the
resolution of the equation
$$\overline{\partial}\sigma_t=-\partial(\varphi(t)\lrcorner\sigma_t)+\varphi(t)\lrcorner\partial\sigma_t$$
by the invertibility of the operators
$e^{i_{\varphi(t)}|i_{\overline{\varphi(t)}}}$ and $\lk \1 -
\overline{\varphi(t)} \varphi(t) \rk^{\!\!-1}\lc$.
 By comparing the
coefficients of $t^k$, it suffices to resolve the system of
equations
\begin{equation}\label{holt-1}
\begin{cases}\db \sigma_0=0,\\
\db\sigma_k = -\partial(\sum_{i=1}^k
\varphi_i\lrcorner\sigma_{k-i}),\qquad
\text{for each $k\geq 1$},\\
\p \sigma_k =0,\qquad \text{for each $k\geq 0$}.
\end{cases}
\end{equation}
By $X_0 \in \mathcal{E}^{p+1,0}$, the equation $\db x=\p \sigma_0$
has solutions, which implies $\p \sigma_0=0$ by type consideration.
Let's resolve \eqref{holt-1} inductively. Since $X_0\in
\mathfrak{D}^{p,1}$, our task is to verify
$$\overline{\partial}\partial(\sum_{i=1}^k\varphi_i\lrcorner\sigma_{k-i})=0$$
for $k\geq 1$. Set
$\eta_k=-\partial(\sum_{i=1}^k\varphi_i\lrcorner\sigma_{k-i})$ for
simplicity. For $k=1$, one has
$$\overline{\partial}\eta_1=-\overline{\partial}\partial(\varphi_1\lrcorner\sigma_0)
=\partial(\overline{\partial}\varphi_1\lrcorner\sigma_0 +
\varphi_1\lrcorner\overline{\partial}\sigma_0)=0,$$ since
$\db\varphi_1=0$ by (\ref{int}) and $\overline{\partial}\sigma_0=0$.
Thus $\sigma_1$ is got by $X_0\in \mathfrak{D}^{p,1}$. By induction,
we assume that \eqref{holt-1} is solved for all $k\leq l$ and thus
have ${\partial}\sigma_k=0$ for $0\leq k\leq l$. By $X_0\in
\mathfrak{D}^{p,1}$, we only need to show
$\overline{\partial}\eta_{l+1}=0$. We resort to a useful commutative
formula (cf. \cite{T,To89,BK,F,C,Li,LR,lry}) on a complex manifold
$X$. For $\phi, \psi\in A^{0,1}(X,T^{1,0}_{X})$ and $\alpha\in
A^{*,*}(X)$,
\[
[\phi,\psi]\lrcorner\alpha=-\p(\psi\lrcorner(\phi
\lrcorner\alpha))-\psi\lrcorner(\phi \lrcorner\p\alpha)
+\phi\lrcorner\p(\psi\lrcorner\alpha)+\psi
\lrcorner\p(\phi\lrcorner\alpha).\] Hence, by this formula and
\eqref{int}, one has
\begin{align*}
\overline{\partial}\eta_{l+1}
 &
 =\partial\lk\sum_{i=2}^{l+1}\overline{\partial}\varphi_i\lrcorner\sigma_{l+1-i}+\sum_{i=1}^{l}\varphi_i\lrcorner\overline{\partial}\sigma_{l+1-i}\rk\\
 &=\partial\lk\frac{1}{2}\sum_{i=2}^{l+1}\sum_{j=1}^{i-1}[\varphi_j,\varphi_{i-j}]\lrcorner\sigma_{l+1-i}-\sum_{i=1}^{l}\varphi_i\lrcorner{\partial}\lk\sum_{j=1}^{l+1-i}\varphi_{j}\lrcorner\sigma_{l+1-i-j}\rk\rk\\
 &=\partial\left(\frac{1}{2}\sum_{i=2}^{l+1}\sum_{j=1}^{i-1}\Big(-\partial\lk\varphi_{j}\lrcorner\lk\varphi_{i-j}\lrcorner\sigma_{l+1-i}\rk\rk
 -\varphi_{j}\lrcorner\varphi_{i-j}\lrcorner\partial\sigma_{l+1-i}\right. \\
  &\qquad \qquad \qquad +\varphi_j\lrcorner\partial\lk\varphi_{i-j}\lrcorner\sigma_{l+1-i}\rk+\varphi_{i-j}\lrcorner\partial\lk\varphi_j\lrcorner\sigma_{l+1-i}\rk\Big)
 \\ &\quad\quad \quad -\sum_{i=1}^{l}\varphi_{i}\lrcorner\partial \left.\lk\sum_{j=1}^{l+1-i}\varphi_j\lrcorner\sigma_{l+1-i-j}\rk \right)\\
 &=\partial\lk\sum_{1\leq j<i\leq
 l+1}\varphi_j\lrcorner\partial\lk\varphi_{i-j}\lrcorner\sigma_{l+1-i}\rk-\sum_{i=1}^{l}\sum_{j=1}^{l+1-i}\varphi_{i}\lrcorner\partial\lk\varphi_j\lrcorner\sigma_{l+1-i-j}\rk\rk\\
 &=0.
\end{align*}
\end{proof}

The proof of Theorem \ref{0q} is a bit different from that of
Theorem \ref{p0} and we need:
\begin{lemma}[\cite{p2}, {Lemma 3.1}]\label{lemma1}
Each Dolbeault class $[\alpha]$ of type $(p,q)$ on a complex
manifold  $X\in \mathfrak{B}^{p+1,q}$ can be represented by a
$d$-closed $(p,q)$-form $\gamma_{\alpha}$.
\end{lemma}

\begin{lemma}\label{lemma2}
Let $\gamma_{\alpha_1}$ and $\gamma_{\alpha_2}$ be two $d$-closed
representatives  of the same Dolbeault class $[\alpha_1]=[\alpha_2]$
as in the above lemma on $X\in \mathcal{E}^{q,0} \cap
\mathfrak{B}^{1,q} $. Then $\gamma_{\alpha_1}=\gamma_{\alpha_2}.$
\end{lemma}
\begin{proof}
From $\gamma_{\alpha_i}=\alpha_i+\db \beta_{\alpha_i},\, i=1,2,$
there exists some $\beta \in A^{0,q-1}(X)$ such that
$$\gamma_{\alpha_2}-\gamma_{\alpha_1} = \db\beta.$$
Since $\gamma_{\alpha_1},\gamma_{\alpha_2}$ are $d$-closed, we have
$\p\db\beta=0$. Hence, by $X \in \mathcal{E}^{q,0}$, the equation
\[ \db x = \p\bar{\beta} \]
has solutions. From type consideration, $\p \bar{\beta}=0$, which
implies $\gamma_{\alpha_1}=\gamma_{\alpha_2}$.
\end{proof}

We shall construct a correspondence from $H^{0,q}(X_0)$ to
$H^{0,q}(X_t)$ by sending $[\alpha]\in H^{0,q}(X_0)$ to
$[e^{i_{\overline{\varphi}}}(\gamma_{\alpha}(t))]\in H^{0,q}(X_t)$,
where
$$\gamma_{\alpha}(t)=\gamma_{\alpha}+\sum_{k=1}^\infty
\gamma_{\alpha}^k \bar{t}^k\in A^{0,q}(X_0).$$ Here
$\gamma_{\alpha}$ is uniquely determined by the Dolbeault class
$[\alpha]$ from the above two lemmas. To guarantee that this
correspondence can not send a nonzero class in $H^{0,q}(X_0)$ to a
zero class in $H^{0,q}(X_t)$, one needs $h^{0,q-1}_t=h^{0,q-1}_0$.
Therefore, for each $1\leq q\leq n$, we use induction to reduce
Theorem \ref{0q} to the following proposition  with all $1\leq
q'\leq q$.
\begin{proposition}\label{thm0q}
Let
$X_0\in\mathfrak{B}^{1,q'}\cap\mathcal{E}^{q',0}\cap\mathfrak{D}^{q',1}$.
Then for any $d$-closed $(0,q')$-form $\sigma_0$ on ${X_0}$, there
exits a power series on $X_0$
$$\sigma_t=\sigma_0+\sum_{k=1}^\infty \bar{t}^k\sigma_k\in A^{0,q'}(X_0)$$ such that
$e^{i_{\overline{\varphi(t)}}}(\sigma_t)\in A^{0,q'}(X_t)$ is
$\db_t$-closed with respect to the complex structure on $X_t$.

\end{proposition}
\begin{proof}[Sketch of Proof]
By Corollary \ref{0q-extension}, the invertibility of the operators
$e^{i_{\varphi(t)}|i_{\overline{\varphi(t)}}}$ yields that the
desired $\db_t$-closed condition is equivalent to the resolution of
the equation

$$\lk \1 - \overline{\varphi(t)}\varphi(t)
   \rk^{\!\!-1}\lc\db\sigma_t-q'\db
 \sigma_t-\lk \lk \1 -\overline{\varphi(t)} \varphi(t)
 \rk^{\!\!-1}\lc{\varphi(t)}\rk\lc(\p  \sigma_t+\db
 (\overline{\varphi(t)}\lc\sigma_t))=0.$$
By comparing the coefficients of $\bar{t}^k$, it suffices to resolve
the  system of equations
\[ \begin{cases}\overline{\partial}\sigma_t=0,\\
\p  \sigma_t+\db
 (\overline{\varphi(t)}\lc\sigma_t)=0,
\end{cases}\]
or equivalently, by conjugation,
\begin{equation}\label{holt-2}
\begin{cases}d\sigma_0=0,\\
\overline{\partial}\overline{\sigma_k}=-\partial(\sum_{i=1}^k\varphi_i\lrcorner\overline{\sigma_{k-i}}),\qquad
\text{for each $k\geq 1$},\\
{\partial}\overline{\sigma_k}=0,\qquad \text{for each $k\geq 1$}.
\end{cases}
\end{equation}
Hence, analogously to the proof of Proposition \ref{thmp0}, we are
able to resolve \eqref{holt-2} inductively by the assumption on
$X_0$ and Lemmata \ref{lemma1}, \ref{lemma2}.
\end{proof}

\textbf{Acknowledgement}: We would like to thank Prof. K. Liu for
everything related to this work and our mathematical growth, and
also Dr. Jie Tu for many helpful discussions. This work started when
the second author was invited by Prof. J.-A. Chen to Taiwan
University in May-July 2013 with the support of the National Center
for Theoretical Sciences. Rao is also supported by the National
Natural Science Foundation of China, No. 11301477 and China
Scholarship Council/University of California, Los Angeles Joint
Scholarship Program.

\vspace{1cm}

\end{document}